\documentclass{article}

\usepackage{graphicx}
\usepackage{amsfonts, amsmath, amssymb, amsthm}
\usepackage{nicematrix}
\usepackage{nicefrac}
\usepackage[shortlabels]{enumitem}

\usepackage[style=numeric]{biblatex}
\def\R{\mathbb{R}}
\def\Z{\mathbb{Z}}
\def\del{\partial}
\DeclareMathOperator{\sym}{Sym}
\DeclareMathOperator{\codim}{codim}
\DeclareMathOperator{\sing}{sing}
\DeclareMathOperator{\Ima}{Im}

\newtheorem{theorem}{Theorem}[section]
\newtheorem{corollary}{Corollary}[theorem]
\newtheorem{lemma}[theorem]{Lemma}
\theoremstyle{remark}
\newtheorem*{remark}{Remark}

\title{A Local Condition for Totally Skew Embeddings}
\author{Zachary Norfolk}
\date{October 2024}

\begin{document}

\maketitle

\begin{abstract}
    We introduce a third-order differential condition, analogous to nonzero torsion of a curve, which guarantees a submanifold of Euclidean space is totally skew in a small neighborhood. This condition is used to construct improved totally skew embeddings of $\R^n$, and to solve the totally skew embedding problem for $\R^n$ with $n$ a power of 2. Some algebraic and geometric properties of this condition are also discussed. 
\end{abstract}

\section{Introduction}
Let $M$ be a submanifold of $\R^N$ and $p,q$ be distinct points in $M$.
A pair of tangent lines at $p$ and $q$ consists of a line tangent to $M$ at $p$, and a line tangent to $M$ at $q$. This pair is skew if these tangent lines are not parallel, and do not intersect. If all  pairs of tangent lines at $p$ and $q$ are skew, we say that $M$ is totally skew at $p$ and $q$. 

An embedding $f:M \to \R^N$ is a totally skew embedding if $f(M)$ is totally skew at all $p, q \in f(M)$ with $p\neq q$. Totally skew embeddings were first studied in \cite{Ghomi} with the goal being to determine the smallest dimension Euclidean space needed for a totally skew embedding of $M$ to exist. Following the notation of \cite{Ghomi}, we denote this minimum dimension by $N(M)$. The problem of computing $N(M)$ is in general difficult, and had been done for only three manifolds:
$$N(\R^1) = 3 \hspace{2em} N(S^1) = 4 \hspace{2em} N(\R^2) = 6.$$
In Corollary \ref{sol}, we provide a solution of $N(\R^n)$ when $n$ is a power of 2, but the general result still remains unknown even for Euclidean spaces. On the other hand, much more is known about upper and lower bounds on $N(M)$. Methods for computing these bounds using topological properties of $M$ were introduced in \cite{Baralic}, \cite{Blagojevic}, and \cite{Ghomi}, and were applied to a variety of manifolds. The simplest example of such a bound is given by $2n+1 \leq N(M) \leq 4n+1$ where $n = \dim(M)$. 

Replacing the condition that lines tangent at $p$ and lines tangent at $q$ are skew with the weaker condition that they are only non-parallel leads to the more recent notion of totally non-parallel immersions. These were introduced and studied in \cite{Harrison}, with a key result being the existence of a second-order differential condition ensuring a manifold is totally non-parallel in small neighborhoods. The question of whether such a local condition exists for totally skew embeddings was asked and briefly discussed in \cite{Harrison}, but left as an open problem.

\subsection{Summary of Results}

The main result of this paper is to show that this local condition does exist, i.e. that there is a third-order differential condition ensuring a manifold is totally skew in small neighborhoods. The statement of this condition and its proof may be found in Section 2. This local condition can be used to construct new examples of totally skew embeddings of $\R^n$, which we carry out in Section 3. Except for the cases $n = 1$ or $2$, these constructions require fewer dimensions in the codomain than the previously best known constructions of \cite{Ghomi}. In Section 4 we present a geometrical reformulation of the local condition analogous to the torsion of a curve in $\R^3$. In fact, for curves in $\R^3$, the local condition is equivalent to the requirement of nonzero torsion. We also show that a generic smooth map $f:M \to \R^{4n}$ satisfies the local condition at all points, by examining the space of degree 3 polynomials. This is an analogue to Proposition 2.3 of \cite{Ghomi} that a generic smooth map $f:M \to \R^{4n+1}$ is totally skew, as well as Theorem 2.6 of \cite{Harrison} that a generic smooth map $f:M \to \R^{4n-1}$ satisfies the second-order (non-parallel) local condition at all points. 

\subsection{Acknowledgements}
I am very grateful to Sergei Tabachnikov for his many valuable suggestions and comments, and for introducing me to the problem. 
I would also like to thank Masato Tanabe for a helpful discussion related to Section 4.2. 
This work was completed with partial support from NSF grant DMS-2005444 and partial support from the Anatole Katok Center for Dynamical Systems and Geometry.

\section{The Local Condition}
Let $f: \R^{n} \to \R^N$ be a smooth map, and fix some $a \in \R^n$. We will use $D^kf_a \in \hom(\sym^k(\R^n), \R^N)$ to denote the symmetric $k$-linear map of $k^{th}$ order partial derivatives of $f$. On the standard basis $\{e_1, \hdots, e_n\}$, this is defined by 
$$D^kf_a(e_{i_1}, \hdots, e_{i_k}) := \frac{\del^k f}{\del_{x_{i_1}} \hdots \del_{x_{i_k}}}(a).$$
In particular, $D^1f_a$ is just the usual derivative of $f$ at $a$, denoted $Df_a$.
\begin{theorem}
\label{lc}
    Suppose $Df_a$, $D^2f_a$, and $D^3f_a$ have the property that solutions $(v_1, v_2, v_3, \lambda) \in \R^n \times \R^n \times \R^n \times \R$ of the equation
    $$Df_a(v_1) +D^2f_a(v_2, v_3) + \lambda D^3f_a(v_3,v_3,v_3) = 0$$
    with $v_3 \neq 0$ must have $v_1 = 0, v_2 = 0$ and  $\lambda = 0$. Then $f$ is a totally skew embedding when restricted to some open neighborhood of $a$. 
\end{theorem}

If we fix $\lambda = 0$ so that the third-order term is not present, we recover something equivalent to the local condition of \cite{Harrison} mentioned in the introduction. This second-order condition requires  $D^2f_a$ to be nonsingular:
\begin{align*}
    D^2f_a(x, y) = 0 \implies x = 0 \text{ or } y = 0,
\end{align*}
and $\Ima Df_a \,\cap \,\Ima D^2f_a = \{0\}$. If this condition is satisfied, then locally there will be no parallel tangent lines near $a$. For intersecting tangent lines the situation is much different as no second-order condition can prevent their existence. For if such a condition exists, we may find some $L \in \hom(\R^n, \R^N)$ and $B\in \hom(\sym^2(\R^n), \R^N)$ such that the map
$$x\mapsto L(x) + B(x,x)$$
satisfies this condition at $x =0$. But for any nonzero $v \in \R^n$, the curve 
$$t \mapsto L(tv) + B(tv, tv)$$
is contained in the plane spanned by $L(v)$ and $B(v,v)$, and any planar curve has intersecting tangent lines. As we will later see, the third-order part of Theorem \ref{lc} prevents these planar curves from existing, essentially by causing the first three derivatives of any curve to be linearly independent. 

\begin{remark}
A $T$-embedding is another variant on totally skew embeddings which comes from replacing the requirement that tangent lines at $p$ and tangent lines at $q$ are skew with the weaker condition that they are only non-intersecting (see \cite{Ghomi2}, \cite{Stojanovic}). It is unclear if a weaker local condition for $T$-embeds exists; one that does not necessarily prevent parallel lines. The fact that two $n$-planes without  intersecting or parallel lines remain without intersecting or parallel lines after small perturbations is needed in the proof of Theorem \ref{lc}. If we only require the $n$-planes to not intersect, this is no longer an open condition.   
\end{remark}

\subsection{Proof of the Local Condition}
We will associate to $f:\R^n \to \R^N$ a map $$F: (\R^n \times \R^n - \Delta) \to \hom(\R^{2n+1}, \R^N)$$ where $\Delta$ denotes the diagonal and $\hom(\R^{2n+1}, \R^N)$ the space of linear transformations $\R^{2n+1}\to \R^N$. Identifying  $\R^{2n+1}$ with 
$\R^n \oplus \R^n \oplus \R$ gives us three coordinates on $\hom(\R^{2n+1}, \R^N)$, coming from the splitting
$$\hom(\R^{2n+1}, \R^N) \cong \hom(\R^{n}, \R^N) \oplus \hom(\R^n, \R^N) \oplus \R^N.$$
Using these coordinates, we define
$$F(p, q) := (Df_p, Df_q, f(q)-f(p)).$$
Explicitly, $F(p,q)$ acts on $(v_1, v_2, \lambda) \in \R^{2n+1}$ by 
\begin{align}
\label{eq1}
    (v_1, v_2, \lambda) \mapsto Df_p(v_1) + Df_q(v_2) + \lambda(f(q) - f(p)).
\end{align}

We will let $\mathcal{U} \subset \hom(\R^{2n+1}, \R^N)$ denote the linear transformations with rank $2n+1$. 

\begin{lemma}
\label{lem1}
    Assume $f$ is an embedding of some open subset $V$. Then $f(V)$ is totally skew at $f(p)$ and $f(q)$ if and only if  $F(p,q) \in \mathcal{U}$.
\end{lemma}

\begin{proof}
    Identify the tangent spaces at $f(p)$ and $f(q)$ with the images of the affine  maps $f(p) + Df_p(\cdot)$ and $f(q) + Df_{q}(\cdot)$. These tangent spaces contain a parallel line if and only if $Df_p(v_1) = Df_{q}(v_2)$ has a nontrivial solution for $(v_1, v_2)$. They contain an intersecting line if and only if $f(p) + Df_p(v_1) = f(q) + Df_q(v_2)$ has any solution for $(v_1, v_2)$. 
    
    Therefore, a nonzero vector in the kernel of $F(p,q)$ (see (\ref{eq1})) with $\lambda = 0$ corresponds to parallel tangent lines, and one with $\lambda \neq 0$ can be re-scaled to have $\lambda = 1$, which corresponds to intersecting tangent lines.
\end{proof}

The proof of Theorem \ref{lc} now consists of two main steps:  a ``spherical blow-up" of $\Delta \subset \R^n \times \R^n$ and modification of $F$ to be defined on the blow-up. The local condition is then related to how this modified function is defined on the boundary of the blow-up. 

\subsubsection{Spherical Blow-Up}
We define the spherical blow-up of $\Delta$
\begin{align*}
    Bl_\Delta := \R^n \times S^{n-1} \times \R_{\geq 0} \hspace{4em}\Phi:Bl_\Delta &\to \R^n \times \R^n
\end{align*}
by $\Phi(a,y,t) = (a, a+ty)$.
The fiber over any point in $\Delta$ is a sphere $S^{n-1}$, while outside of $\Delta$, $\Phi$ gives a diffeomorphism between $\R^n \times S^{n-1} \times \R_{>0}$ and $\R^n \times \R^n - \Delta$. We will modify the function 
$$F\circ\Phi(a,y,t) = (Df_a, Df_{a+ty}, f(a+ty)-f(a))$$
 to be compatible with the Taylor series expansion of $f$. This will be defined $\widetilde F(a,y,t) :=$
\begin{align*}
    \left(Df_a, \frac{Df_{a+ty} - Df_a}{t}, \frac{6(Df_{a+ty}(ty)+Df_a(ty))-12(f(a+ty) - f(a))}{t^3} \right).
\end{align*}

\begin{lemma}
\label{lem2}
    For any $(a,y,t) \in Bl_\Delta$ with $t\neq 0$, we have that $F \circ \Phi(a,y,t) \in \mathcal{U}$ if and only if $\widetilde F(a,y,t) \in \mathcal{U}$. 
\end{lemma} 

\begin{proof}
    Let $A_1 = F\circ \Phi(a,y,t)$ and $A_2 = \widetilde F(a,y,t)$. Consider the following $(2n+1) \times (2n+1)$ matrix, given in block form by 
$$
B= \begin{bNiceMatrix}[hvlines,margin]
\Block{2-3}<> {\mathbf{I}} & & & \Block{2-4}<> {\frac{-1}{t}\mathbf{I}} & & & &\Block{2-1}{\frac{6}{t^2} y} \\
& & && &&  &\\
\Block{2-3}<> {\mathbf{0} }& & & \Block{2-4}<> {\hspace{.12cm}\frac{1}{t}\hspace{.1cm}\mathbf{I}} & & & &\Block{2-1}{\frac{6}{t^2} y} \\
& & & & & & &\\
\Block{1-7}{0 \hspace{1.2em}\hdots \hspace{1.2em} 0} & & & & && & \nicefrac{-12}{t^3}
\end{bNiceMatrix}.$$
Here $\mathbf{I}$ and $\mathbf{0}$ are the $n\times n$ identity and zero matrices, and $y \in S^{n-1}$ is treated as an $n\times 1$ matrix of its components. Clearly $B \in GL(\R^{2n+1})$, and one may compute that $A_2 = A_1\circ B$. Thus, $A_1$ and $A_2$ have the same rank. 
\end{proof}

\subsubsection{Extension of $\widetilde F$}
We now address the issue of extending the domain of $\widetilde F$ to the $S^{n-1}$ fibers over $\Delta$. To do this continuously, we use the following formulation of Taylor's theorem (see Theorem 5.2 of \cite{Coleman}). 
\begin{theorem}[Taylor's Theorem]
    Let $f: \R^n \to \R^m$ be a smooth function. Then for any $a \in \R^n$ and $k \in \Z^+$ there is a Taylor expansion
    $$f(a+x) = f(a) + Df_a(x) + \frac{1}{2}D^2f_a(x,x) + \hdots + \frac{1}{k!}D^kf_a (x^{\otimes k}) + R(a,x)$$
    where $$\|R(a,x)\| \leq \frac{\|x\|^{k+1}}{(k+1)!} \sup_{0\leq s\leq 1}\|D^{k+1}f_{a+s x}\|.$$
\end{theorem}
Here $\|D^kf_a\|$ is the operator norm given by
$$\|D^kf_a\| := \sup_{x_1, \hdots, x_k\in S^{n-1}} \|D^kf_a(x_1, \hdots, x_k)\|.$$

\begin{lemma}
\label{lem3}
    With the same notation as above, the function $\widetilde{R}: \R^n \times \R^n \to \R^m$ given by 
    \begin{align*}
        \widetilde{R}(a,x) = 
        \begin{cases}
            \frac{R(a,x)}{\|x\|^k} &x \neq 0\\
            0 & x = 0
        \end{cases}
    \end{align*}
    is continuous.  
\end{lemma}
\begin{proof}
Continuity just needs to be checked when $x = 0$. For $p \in \overline B_{2r}(a)$, the closed ball of radius $2r$ with center $a$, $\|D^{k+1}f_p\|$ is bounded above by a constant $c$, as $D^{k+1}f_p$ is continuous as a function of $p$. Thus, for $(a', x') \in B_r(a) \times (B_r(0)-\{0\})$  we have that $$\|\widetilde{R}(a',x')\| = \frac{\|R(a', x')\|}{\|x'\|^{k}} \leq \frac{c\|x'\|}{(k+1)!}. $$ This shows continuity at $(a,0)$. 
\end{proof}

We will address the issue of extending $\widetilde F$ by dealing with each of its three components individually. The first component is well defined on all of $Bl_\Delta$. The second component will be modified by substituting the first-order Taylor expansion of $Df: \R^n \to \hom(\R^n, \R^N)$ and using the identification of $D(Df) \in \hom(\R^n, \hom(\R^n, \R^N))$ with $D^2f$: 
\begin{align*}
        Df_{a+x}(\cdot) &= Df_a(\cdot) + D^2f_a(x,\cdot) + R_1(a,x)(\cdot).
\end{align*}
The third component will be modified by substituting both the second-order expansion of $Df$ (again with appropriate identification of its higher derivatives with those of $f$), as well as the third-order expansion of $f$:
\begin{align*}
      Df_{a+x}(\cdot) &= Df_a(\cdot) + D^2f_a(x,\cdot) + \frac{1}{2}D^3f_a(x,x,\cdot) + R_2(a,x)(\cdot)\\
      f(a+x) - f(a) &= Df_a(x) + \frac{1}{2}D^2f_a(x,x) + \frac{1}{6}D^3f_a(x,x,x) + R_3(a,x).
\end{align*}
In summary, $\widetilde{F}(a, y, t) = $
\begin{align*}
    \left(Df_a, D^2f_a(y, \cdot) + \frac{R_1(a, ty)}{t}(\cdot), D^3f_a(y,y,y) + \frac{6R_2(a, ty)}{t^2}(y) - \frac{12R_3(a, ty)}{t^3}\right).
\end{align*}
Applying Lemma \ref{lem3}, this extends continuously to all of $Bl_\Delta$ by letting 
$$\widetilde F(a,y,0) = \left(Df_a, D^2f_a(y, \cdot), D^3f_a(y,y,y)\right).$$

\subsubsection{Proof of Theorem \ref{lc}}
\begin{proof}
    Assuming that $N \geq 2n+1$, $\mathcal{U}$ is an open set and thus $\widetilde F^{-1}(\mathcal{U})$ is as well. The local condition being satisfied at $a \in \R^n$ is equivalent to the statement that $\widetilde{F}(a, y, 0) \in \mathcal{U}$ for all $y \in S^{n-1}$. Therefore, by the tube lemma, inside of $\widetilde F^{-1}(\mathcal{U})$ we may find a tube of the form $B_r(a) \times S^{n-1} \times [0, 2r)$ for some $r >0$. The image of this tube under $\Phi$ will contain $B_r(a) \times B_r(a)$. We may assume $f$ is an embedding of $B_r(a)$ by shrinking $r$ if necessary, since the local condition implies $Df_a$ is injective. By Lemmas \ref{lem1} and \ref{lem2}, it follows that $f$ gives a totally skew embedding of $B_r(a)$.
\end{proof}

\section{Construction Using the Local Condition}
We now use Theorem \ref{lc} to construct totally skew embeddings of open disks. Here we present an explicit example of a cubic polynomial $\R^n \to \R^{3n}$ that satisfies the local condition at $x=0$. In the following section, we will see that the set of cubic polynomials $\R^n \to \R^N$ satisfying the local condition at $x=0$ is open and dense when $N \geq 3n$.\\

Let $B: \R^n \times \R^n \to \R^{2n}$ be the symmetric  bilinear map given by 
\begin{align*}
    \begin{bmatrix}
        x_0 \\
        \vdots \\
        x_{n-1}
    \end{bmatrix}
    \times 
    \begin{bmatrix}
        y_0 \\
        \vdots \\
        y_{n-1}
    \end{bmatrix}
    \mapsto 
    \begin{bmatrix}
        0\\
        \sum_{i+j=0} x_i y_j \\
        \vdots \\
        \sum_{i+j= 2n-2} x_i y_j \\
    \end{bmatrix}
\end{align*}

and let $C:\R^n \times \R^n \times \R^n \to \R^{2n}$ be the symmetric trilinear map given  by 
\begin{align*}
    \begin{bmatrix}
        x_0 \\
        \vdots \\
        x_{n-1}
    \end{bmatrix}
    \times 
    \begin{bmatrix}
        y_0 \\
        \vdots \\
        y_{n-1}
    \end{bmatrix}
    \times 
    \begin{bmatrix}
        z_0 \\
        \vdots \\
        z_{n-1}
    \end{bmatrix}
    \mapsto 
    \begin{bmatrix}
        x_0y_0z_0\\
        \vdots\\
        x_{n-1}y_{n-1}z_{n-1}\\
        0 \\
        \vdots\\
        0
    \end{bmatrix}.
\end{align*}
Note that $B$ is nonsingular: $B(x, y) = 0 \implies x = 0 \text{ or } y=0$. 
\begin{theorem}
    The map $f:\R^n \to \R^n \times \R^{2n}$ by $$f(x) = \left(x, \frac{1}{2}B(x,x) + \frac{1}{6}C(x,x,x)\right)$$
    satisfies the local condition at $x = 0$. 
\end{theorem}

\begin{proof}
    We compute the first three derivatives of $f$ at zero:
    \begin{align*}
        Df_0(v_1) &= (v_1, 0)\\
        D^2f_0(v_1, v_2) &= (0, B(v_1, v_2))\\
        D^3f_0(v_1, v_2, v_3) &= (0, C(v_1, v_2, v_3)).
    \end{align*}
    Let $(v_1, v_2, v_3, \lambda)$ be a solution to 
    \begin{align}
    \label{eq2}
    (v_1, 0) + (0, B(v_2, v_3))+ \lambda (0, C(v_3, v_3, v_3)) = (0,0)
    \end{align}
    with $v_3 \neq 0$. Clearly $v_1 = 0$ and if we assume $\lambda = 0$, then $B(v_2, v_3) = 0$ as well. Since $B$ is nonsingular, this would give $v_2 = 0$. To prove the local condition then, it remains to show that there are no solutions to (\ref{eq2}) with $\lambda \neq 0$. For this, we will show that the equation
$$B(x, y) + \lambda C(y, y, y) = 0$$
 with $\lambda \neq 0$ only has solutions when $y = 0$. Letting $x_i$ and $y_i$ denote the components of $x$ and $y$, the components of this equation are
 \begin{align*}
     \begin{bmatrix}
         \lambda y_0^3 \\
         \sum_{i+j = 0}x_i y_j+\lambda y_1^3 \\
          \vdots \\
        \sum_{i+j = n-2} x_iy_j+ \lambda y_{n-1}^3 \\
        \vdots 
     \end{bmatrix} = 
     \begin{bmatrix}
        0 \\ 0 \\ \vdots \\ 0 \\ \vdots
     \end{bmatrix}.
 \end{align*}
 Starting with the first component, since $\lambda \neq 0$ we have $y_0 = 0$. Then assuming inductively that $y_\ell = 0$ for all $ \ell<k\leq n-1$, the component of the equation reading
 $$\sum_{i + j = k-1}x_i y_j + \lambda y_k^3 = 0$$
 reduces to $\lambda y_k^3 = 0$, so that $y_k = 0$ as well. Thus, we see that $y = 0$. 
\end{proof}

We may embed $\R^n$ into the open disk where $f$ is totally skew to get the following upper bound. 

\begin{corollary}
    $N(\R^n) \leq 3n$.
\end{corollary}

For certain values of $n$, this matches the known lower bounds (see Corollary 1.6 of \cite{Ghomi}).

\begin{corollary}
\label{sol}
    $N(\R^n) = 3n$ when $n$ is a power of 2.
\end{corollary}

\section{Properties of the Local Condition}
We conclude with a discussion of some algebraic and geometric aspects of the local condition. 
\subsection{Geometric Interpretation}
Let $M$ be a submanifold of $\R^N$ and $p \in M$. Let $f:U \to \R^N$ be a chart around $p$. The local condition for $f$ being satisfied at $p$ has the following geometric interpretation, which is independent of the coordinate system.

\begin{theorem}
    The local condition for $f$ is satisfied at $p$ if and only if both of the following hold:
    \begin{enumerate}[1)]
    \item The second fundamental form $\mathrm{I\!I}:T_pM\times T_pM \to T_pM^\perp$ is nonsingular as a bilinear map.
    \item Any regular curve $\gamma:(0,1) \to M \subseteq \R^N$ that passes through $p$ with $\gamma(t_0) = p$, has $\gamma'''(t_0) \neq 0$.
    \end{enumerate}
\end{theorem}

For example, in the simplest case where $M$ is a curve in $\R^3$, the conditions above amount to nonzero torsion at $p$. Indeed, choosing a regular parameterization of the curve, $\gamma$ with $\gamma(t_0) = p$, the first condition implies that $\gamma'(t_0)$ and $\gamma''(t_0)$ are linearly independent. If $\gamma'''(t_0)$ lies in this osculating plane, then we could reparameterize $\gamma$ to have vanishing third derivative at $p$. 

Condition \textit{1)} on the second fundamental form is the geometric interpretation of the local condition of \cite{Harrison} preventing parallel tangent lines. A relation of condition \textit{2)} with the prevention of intersecting tangent lines was discussed after the statement of Theorem \ref{lc}.

Note that this theorem also shows that the local condition is coordinate invariant.

\begin{proof}

Let $(v_1, v_2, v_3, \lambda)$ be a solution to  
\begin{equation}
\label{geo}
    Df_p(v_1) +D^2f_p(v_2, v_3) + \lambda D^3f_p(v_3,v_3,v_3) = 0.
\end{equation}
In these local coordinates, $\mathrm{I\!I}$ is given by orthogonal projection of $D^2f_p$ into $(\Ima Df_p)^\perp$. Thus $\mathrm{I\!I}$ is nonsingular if and only if solutions $(w_1, w_2, w_3)$ of 
$$Df_p(w_1) + D^2f_p(w_2, w_3)  = 0$$
must have $w_1, w_2 = 0$ or $w_1, w_3 = 0$. This is equivalent to requiring solutions of (\ref{geo}) with $v_3 \neq 0$ and $\lambda = 0$ to have $v_1, v_2 = 0$. 

Now consider a regular curve with image in $M$. We may treat this curve as a composition of $f$ with a regular curve $\gamma$ into $U$. Denote $x_1 = \gamma'''(t_0), x_2 = 3 \gamma''(t_0)$ and $x_3 = \gamma'(t_0)$. Using Fa\`a di Bruno's formula,
$$(f\circ\gamma)'''(t_0) = Df_p(x_1) + D^2f_p(x_2, x_3) + D^3f_p(x_3, x_3, x_3).$$
This is nonzero for any regular curve if and only if 
there are no solutions of (\ref{geo}) with $v_3 \neq 0$ and $\lambda \neq 0$. 
\end{proof}

\subsection{Stratification}
    Let $\mathcal{Y}$ be the subset of $$W = \hom(\R^n \oplus \sym^2(\R^n) \oplus \sym^3(\R^n), \R^N)$$ with $N \geq 3n$, consisting of triples $(L, B, T)$ of linear, symmetric bilinear, and symmetric trilinear maps for which the function 
    $$x \mapsto L(x) + \frac{1}{2}B(x,x) + \frac{1}{6}T(x,x,x)$$
    fails to satisfy the local condition at $x = 0$. We may think of $W$ as a space of degree 3 polynomials with $L$, $B$, and $T$ corresponding to the first, second, and third derivatives at $x=0$. Then $\mathcal{Y}$ consists of the triples $(L,B,T)$ such that 
    $$L(v_1) + B(v_2, v_3) + \lambda T(v_3, v_3, v_3) = 0$$ 
    for some nonzero $v_3 \in \R^n$ and nonzero $(v_1, v_2,\lambda) \in \R^n \oplus \R^n \oplus \R= \R^{2n+1}$. It is not difficult to see that if $(L, B, T) \in \mathcal{Y}$, then the $v_3$ and $(v_1, v_2, \lambda)$ from above may be rechosen to have unit lengths in $\R^n$ and $\R^{2n+1}$. 

    \begin{theorem}
        $\mathcal{Y}$ is a closed, semialgebraic subset of $W$ with codimension $N - 3n+1$. 
    \end{theorem}

    \begin{proof}
        Consider the expression 
        \begin{align*}
            L(v_1) + B(v_2, v_3) + \lambda T(v_3, v_3, v_3)
        \end{align*}
        as an $\R^N$ valued function of the variables 
        \begin{center}$(L, B, T) \in W$, $(v_1, v_2, \lambda) \in \R^{2n+1} = \R^n \oplus \R^n \oplus \R$, and $v_3 \in \R^n$.
        \end{center}
        We will denote this function $g:W \times \R^{2n+1} \times \R^n \to \R^N$. Holding the $\R^{2n+1}$ and $\R^n$ inputs constant gives a linear function $W \to \R^N$. This function is surjective when both of the fixed $\R^{2n+1}$ and $\R^n$ inputs are nonzero since the three evaluation maps 
        \begin{align}
            \label{s1}L&\mapsto L(v_1) \\ 
            \label {s2}B &\mapsto B(v_2, v_3)\\
            \label{s3}T&\mapsto \lambda T(v_3, v_3, v_3)
        \end{align}
        are surjective for $v_1 \neq 0$ (\ref{s1}) , $v_2, v_3 \neq 0$ (\ref{s2}), and $\lambda, v_3 \neq 0$ (\ref{s3}).

        This argument implies that, with the restricted domain of $W \times S^{2n} \times S^{n-1}$, $0$ is a regular value of the function $g$. Thus, $g^{-1}(0)$ is a submanifold of codimension $N$, and its projection onto $W$ is just $\mathcal{Y}$. The three claims now follow from properties of $g^{-1}(0)$. 

        That $\mathcal{Y}$ is closed follows since $g^{-1}(0)$ is closed, and compactness of $S^{2n} \times S^{n-1}$ means the projection map is closed. In a suitable basis of $W$ we can represent $g$ as a polynomial map, so that $g^{-1}(0)$ is a real algebraic variety. That $\mathcal{Y}$ is semialgebraic then follows from the Tarski–Seidenberg theorem. That $\codim (\mathcal{Y}) \geq N-3n+1$ follows from  $g^{-1}(0)$ having codimension $N$. To turn this into an equality, it suffices to find a point in $g^{-1}(0)$ where the differential of the projection is injective. Construction of such a point is done in the Appendix.      
    \end{proof}
    
\begin{corollary}
    Let $n = \dim({M})$ and $N\geq 4n$. Then the set of $f \in C^\infty(M, \R^N)$ satisfying the local condition at all points is open and dense (in the strong topology).
\end{corollary}

\begin{proof}
    Let $\mathcal{Z} \subseteq J^3(M, \R^N)$ consist of the 3-jets of maps failing the local condition at their source. This set is well defined, as the local condition is coordinate invariant. For each coordinate chart $U_\alpha$ on $M$, the corresponding trivialization $U_\alpha \times \R^N \times W$ of the jet bundle intersects $\mathcal{Z}$ in sets of the form $U_\alpha \times \R^N \times \mathcal{Y}$. Since $\mathcal{Y}$ is closed, so is $\mathcal{Z}$, and the set of functions whose 3-jet extension is disjoint from $\mathcal{Z}$ is open. 
    
    Let $ \sing(\mathcal{Y})$ denote the set of singular points of $\mathcal{Y}$,  which is a closed semialgebraic subset of strictly smaller dimension (see I.2.9.5 of \cite{Shiota} or $\S 2$ and $\S 4$ of \cite{Mather}). This is the complement of the set of regular points: the points in a neighborhood of which, $\mathcal{Y}$ is a smooth submanifold of dimension $\dim(\mathcal{Y})$.  We may inductively define $\sing^{k+1}(\mathcal{Y}) = \sing(\sing^{k}(\mathcal{Y}))$ by setting $\sing^0(\mathcal{Y}) = \mathcal{Y}$, and only finitely many of these sets will be nonempty. The sets $\mathcal{Y}_k = \sing^{k}(\mathcal{Y}) - \sing^{k+1}(\mathcal{Y})$ then give a decomposition of $\mathcal{Y}$ into a finite union of submanifolds with dimensions no larger than $\dim(\mathcal{Y})$. 
    
    To turn this decomposition of $\mathcal{Y}$ into one of $\mathcal{Z}$, we must show that each $\mathcal{Y}_k$ is invariant under the transition maps of the jet bundle.  On the overlap of two local trivializations of $J^3(M, \R^N)$ coming from charts $U_\alpha$ and $U_\beta$ in $M$, the transition map gives automorphisms of $W$ induced by the change of coordinate map in $U_\alpha \cap U_\beta$. Since the local condition is coordinate invariant, $\mathcal{Y}$ will be invariant under these automorphisms. Therefore, $\sing(\mathcal{Y})$ is also invariant; so by induction on $\sing^k(\mathcal{Y})$,  each $\mathcal{Y}_k$ is invariant as well.
    
    Thus, these fiberwise decompositions extend to a well-defined decomposition of  $\mathcal{Z}$ into a finite union of submanifolds with codimensions no smaller than $\codim(\mathcal{Y}) = N-3n+1 \geq n+1$. The statement now follows by the jet transversality theorem (see 2.9 in Chapter 3 of \cite{Hirsch}).
\end{proof}

\section{Appendix}
Here we finish the proof of Theorem 4.2 that $\codim(\mathcal{Y}) = N-3n+1$ by constructing a point in $g^{-1}(0)$ where the differential of the projection $W\times S^{2n} \times S^{n-1} \to W$ is injective. 

 A vector tangent to the fiber of the projection consists of a $(w_1, w_2, \nu) \in \R^{2n+1}$ and a $w_3 \in \R^{n}$ with 
 \begin{align}
 \label{perp2}
 (v_1, v_2, \lambda) &\perp (w_1, w_2, \nu)\\
     \label{perp1}
     v_3 &\perp w_3 .
 \end{align}
 This vector is in the tangent space of $g^{-1}(0)$ if and only if it is in the kernel of the differential of $g$: 
        \begin{align}
        \label{diff}
            L(w_1) + B(w_2, v_3) + B(v_2, w_3) + \nu T(v_3, v_3, v_3) + 3\lambda T(w_3, v_3, v_3) = 0.
        \end{align}
        A point in $g^{-1}(0)$ consists of a triple $(L,B,T)$ failing the local condition at unit vectors $(v_1, v_2, \lambda), v_3$. Thus, our goal is to construct such a point with the additional property that solutions to (\ref{diff}) satisfying (\ref{perp2}) and (\ref{perp1}) must have $(w_1, w_2, \nu)$ and $w_3$ both zero. 

Let $\overline B: \R^n \times \R^n \to \R^{2n}$ be the symmetric bilinear map defined on the standard basis $\{e_1,\hdots,e_n\}$ of $\R^n$ and $\{e'_1,\hdots, e'_{2n}\}$ of $\R^{2n}$ by 
\begin{align*}
    \overline B(e_i, e_j) &= e'_{i+j}\text{ for $i \leq j$ and $(i,j) \neq (1, n)$}\\
    \overline B(e_1, e_n) &= 0.
\end{align*}
Treating $\R^N$ as $\R^n \oplus \R^{2n} \oplus \R^{N-3n}$, we define
\begin{align*}
    L(x) &= (x, 0,0)\\
    B(x,y) &= (0, \overline{B}(x,y), 0)\\
    T(x,y,z) &= (0, C(x,y,z), 0)
\end{align*}
where $C$ is the symmetric trilinear map of Section 3. 
It is easy to verify that $(L, B, T)$ fails the local condition at $(v_1, v_2, \lambda)= (0,e_n, 0)$, $v_3 = e_1$. 

Using these values, we can rewrite (\ref{diff}) as
\begin{align*}
    (w_1, \overline B(w_2, e_1)+ \overline B(e_{n}, w_3) + \nu C(e_1, e_1, e_1), 0) = (0,0,0).
\end{align*}
Then $w_1 = 0$, and using the definition of $C$ we may verify that $$C(e_1, e_1, e_1) \perp \overline B(x,y) \text{ for all }x, y.$$ Therefore $\nu = 0$ and $\overline B(w_2, e_1) + \overline B(e_n, w_3) = 0$. 

Now, the image of $\overline B(e_1, \cdot)$ is spanned by $e_2, \hdots, e_{n}$, while the image of $\overline B(e_n, \cdot)$ is spanned by $e_{n+2}, \hdots, e_{2n}$. Thus, $\overline B(e_1, w_2) + \overline B(e_n, w_3) = 0$ implies that $\overline B(e_1, w_2) = 0$ and  $\overline B(e_n, w_3) = 0$, so that $w_2$ is a scalar multiple of $e_n$ and $w_3$ is a scalar multiple of $e_1$. But then (\ref{perp2}) and (\ref{perp1}) imply that $w_2= w_3 = 0$, as desired.

\printbibliography

\end{document}